\documentclass[12pt]{amsart}
\usepackage{amssymb,amsmath,latexsym,amsthm,enumerate,color,graphicx}
\numberwithin{equation}{section}
\newtheorem{thm}{Theorem}[section]
\newtheorem{lem}[thm]{Lemma}

\newtheorem{prop}[thm]{Proposition}
\theoremstyle{remark}
\newtheorem*{definition}{Definition}

\newtheorem*{note}{Note}
\newcommand{\bsk}{\par\vspace{\baselineskip}}
\newcommand{\Z}{\mathbb{Z}}
\newcommand{\noi}{\noindent}

\newcommand{\ord}{\mathrm{ord}}

\newcommand{\C}{\mathbb{C}}
\newcommand{\Q}{\mathbb{Q}}
\newcommand{\F}{\mathbb{F}}

\newcommand{\G}{\Gamma}

\renewcommand{\H}{\mathbb{H}}
\newcommand{\X}{X_0^+(p)}

\renewcommand{\mod}[1]{\,(\mathrm{mod}\,{#1})}

\newcommand{\inv}{^{-1}}
\newcommand{\wt}{\mathrm{wt}}
\newcommand{\leg}[2]{\left(\frac{#1}{#2}\right)}

\begin{document}

\title{Weierstrass points on $X_0^+(p)$ and supersingular $j$-invariants}

\author{Stephanie Treneer}

\address{Department of Mathematics\\
Western Washington University\\
Bellingham, WA 98225}

\email{stephanie.treneer@wwu.edu}

\noi
\begin{abstract} We study the arithmetic properties of Weierstrass points on the modular curves $X_0^+(p)$ for primes $p$.  In particular, we obtain a relationship between the Weierstrass points on $X_0^+(p)$ and the $j$-invariants of supersingular elliptic curves in characteristic $p$.
\end{abstract}


\maketitle

\section{Introduction}

A \emph{Weierstrass point} on a compact Riemann surface $M$ of genus $g$ is a point $Q\in M$ at which some holomorphic differential $\omega$ vanishes to order at least $g$. Weierstrass points can be identified by observing their weight.  Let $\mathcal{H}^1(M)$ be the $\C$-vector space of holomorphic differentials on $M$ of dimension $g$. If $\{\omega_1,\omega_2, \dots, \omega_g\}$ forms a basis for $\mathcal{H}^1(M)$ adapted to $Q\in M$, so that
\[0=\mbox{ord}_Q(\omega_1)<\mbox{ord}_Q(\omega_2)<\cdots<\mbox{ord}_Q(\omega_g),\]
then we define the \emph{Weierstrass weight} of $Q$ to be
\[\wt(Q):=\sum_{j=1}^g(\mbox{ord}_Q(\omega_j)-j+1).\]
We see that $\wt(Q)>0$ if and only if $Q$ is a Weierstrass point of $M$.  The Weierstrass weight is independent of the choice of basis, and it is known that
\[\sum_{Q\in M}\wt(Q)=g^3-g.\]
Hence each Riemann surface of genus $g\geq 2$ must have Weierstrass points. For these and other facts, see Section III.5 of \cite{FK}.

We will consider Weierstrass points on modular curves, a class of Riemann surfaces which are of wide interest in number theory.  Let $\H$ denote the complex upper half-plane.  The modular group $\G:=\mathrm{SL}_2(\Z)$ acts on $\H$ by linear fractional transformations $\left(\begin{smallmatrix}a&b\\c&d\end{smallmatrix}\right)z=\frac{az+b}{cz+d}$.  If $N\geq 1$ is an integer, then we define the congruence subgroup
\[\G_0(N):=\left\{\left(\begin{array}{cc}a&b\\c&d\end{array}\right)\in\G:c\equiv 0\pmod{N}\right\}.\]
The quotient of the action of $\G_0(N)$ on $\H$ is the Riemann surface $Y_0(N):=\G_0(N)\backslash\H$, and its compactification is $X_0(N)$.  The modular curve $X_0(N)$ can be viewed as the moduli space of elliptic curves equipped with a level $N$ structure. Specifically, the points of $X_0(N)$ parameterize isomorphism classes of pairs $(E,C)$ where $E$ is an elliptic curve over $\C$ and $C$ is a cyclic subgroup of $E$ of order $N$.

Weierstrass points on $X_0(N)$ have been studied by a number of authors (see, for example, \cite{LeN}, \cite{Atk}, \cite{Atk2}, \cite{Ogg}, \cite{OggHyp}, \cite{Roh2}, \cite{Roh}, \cite{K1}, \cite{K2}, \cite{AO}, \cite{AP}, and \cite{IJK}). An interesting open question is to determine those $N$ for which the cusp $\infty$ is a Weierstrass point.  Lehner and Newman \cite{LeN} and Atkin \cite{Atk} showed that $\infty$ is a Weierstrass point for most non-squarefree $N$, while Atkin \cite{Atk2} proved that $\infty$ is not a Weierstrass point when $N$ is prime.

Most central to the present paper is the connection between Weierstrass points and supersingular elliptic curves.  Ogg \cite{Ogg} showed that for modular curves $X_0(pM)$ where $p$ is a prime with $p\nmid M$ and with the genus of $X_0(M)$ equal to $0$, the Weierstrass points of $X_0(pM)$ occur at points whose underlying elliptic curve is supersingular when reduced modulo $p$. So in particular, $\infty$ is not a Weierstrass point in these cases, extending \cite{Atk2}.  This has recently been confirmed by Ahlgren, Masri and Rouse \cite{AMR} using a non-geometric proof. Ahlgren and Ono \cite{AO} showed for the $M=1$ case that in fact all supersingular elliptic curves modulo $p$ correspond to Weierstrass points of $X_0(p)$, and they demonstrated a precise correspondence between the two sets.  In order to state their result, we make the following definitions.

For $p$ and $M$ as above, let
\[F_{pM}(x):=\prod_{Q\in Y_0(N)}(x-j(Q))^{\wt(Q)},\]
where $j(z)=q\inv+744+196884q+\cdots$ is the usual elliptic modular function defined on $\Gamma$, and $j(Q)=j(\tau)$ for any $\tau\in\H$ with $Q=\G_0(N)\tau$.  This is the divisor polynomial for the Weierstrass points of $Y_0(N)$.  Next, for a prime $p$ we define
\[S_p(x):=\mathop{\prod_{E/\overline{\F}_p}}_{\mathrm{supersingular}} (x-j(E))\in\F_p[x],\]
where the product is over all $\overline{\F}_p$-isomorphism classes of supersingular elliptic curves.  It is well known that $S_p(x)$ has degree $g_p+1$, where $g_p$ is the genus of $X_0(p)$. Ahlgren and Ono \cite{AO} proved the following.

\begin{thm}\label{AOthm} If $p$ is prime, then $F_p(x)$ has $p$-integral rational coefficients and 
\[F_p(x)\equiv S_p(x)^{g_p(g_p-1)}\pmod{p}.\]
\end{thm}

El-Guindy \cite{ElG} generalized Theorem \ref{AOthm} to those cases where $M$ is squarefree, showing that $F_{pM}(x)$ has $p$-integral rational coefficients and is divisible by $\widetilde{S}_p(x)^{\mu(M)g_{pM}(g_{pM}-1)}$,
where $\mu(M):=[\G:\G_0(M)]$ and $g_{pM}$ is the genus of $X_0(pM)$, and where
\begin{equation}\label{Stilde}\widetilde{S}_p(x):=\mathop{\prod_{E/\overline{\F}_p\;\mathrm{supersingular}}}_{j(E)\ne 0,1728} (x-j(E)).\end{equation}
He also gave an explicit factorization of $F_{pM}(x)$ in most cases where $M$ is prime.  Generalizing Theorem \ref{AOthm} in a different direction, Ahlgren and Papanikolas \cite{AP} gave a similar result for higher order Weierstrass points on $X_0(p)$, which are defined in relation to higher order differentials.

In this paper we consider the modular curve $X_0^+(p)$, the quotient space of $X_0(p)$ under the action of the Atkin-Lehner involution $w_p$, which maps $\tau\mapsto -1/p\tau$ for $\tau\in\H$.  There is a natural projection map $\pi:X_0(p)\to X_0^+(p)$ which sends
a point $Q\in X_0(p)$ to its equivalence class $\pi(Q)=\overline{Q}$ in $X_0^+(p)$. This is a 2-to-1 mapping, ramified at those points $Q\in X_0(p)$ that remain fixed by $w_p$.  Therefore we set
\begin{eqnarray}\label{vq}v(Q):=\begin{cases}2&\mbox{if $w_p(Q)=Q$},\\
1&\mbox{otherwise,}\end{cases}\end{eqnarray}
so that $v(Q)$ is equal to the multiplicity of the map $\pi$ at $Q$.
We now define a divisor polynomial for the Weierstrass points of $\X$. We will set our product to be over $X_0(p)$ rather than $\X$ to preserve the desired $p$-integrality of the coefficients. Let
\[\mathcal{F}_p(x):=\prod_{Q\in Y_0(p)}(x-j(Q))^{v(Q)\wt(\overline{Q})},\]
where $\wt(\overline{Q})$ is the Weierstrass weight of the image $\overline{Q}$ of $Q$ in $X_0^+(p)$. The zeros of this polynomial capture those non-cuspidal points of $X_0(p)$ which map to Weierstrass points in $\X$. The two cusps of $X_0(p)$ at 0 and $\infty$ are interchanged by $w_p$, so that $X_0^+(p)$ has a single cusp at $\infty$, which may or may not be a Weierstrass point.  Atkin checked all primes $p\leq 883$ and conjectured that $\infty$ is a Weierstrass point for all $p> 389$.  Stein has confirmed this for all $p<3000$, and his table of results can be found in \cite{Stein}. Therefore $\mathcal{F}_p(x)$ is a polynomial of degree $2((g_p^+)^3-g_p^+ -\wt(\infty))$, where $g^+_p$ is the genus of $\X$.

Recalling that the supersingular polynomial $S_p(x)$ factors over $\F_p[x]$ into linear and irreducible quadratic factors, we separate these factors by defining
\[S_p^{(l)}(x):=\mathop{\prod_{E/\overline{\F}_p\;\mathrm{supersingular}}}_{j(E)\in\F_p} (x-j(E))\qquad\mbox{ and }\qquad S_p^{(q)}(x):=\mathop{\prod_{E/\overline{\F}_p\;\mathrm{supersingular}}}_{j(E)\in\F_{p^2}\backslash\F_p} (x-j(E)) ,\]
so that $S_p(x)=S_p^{(l)}(x)\cdot S_p^{(q)}(x)$. Our main theorem gives an analogue of Theorem \ref{AOthm} for $\mathcal{F}_p(x)$. We require an assumption that $\mathcal{H}^1(\X)$ has a \emph{good basis}, a condition about $p$-integrality which we define later in Section \ref{basissec}. Computations suggest that most, if not all, such spaces satsify this condition. Indeed, each $\mathcal{H}^1(\X)$ with $p<3200$ has a good basis.

\begin{thm}\label{main} Let $p$ be prime and suppose that $\mathcal{H}(\X)$ has a good basis.  Then $\mathcal{F}_p(x)$ has $p$-integral rational  coefficients, and there exists a polynomial $H(x)\in\mathbb{F}_p[x]$ such that
\[\mathcal{F}_p(x)\equiv S_p^{(q)}(x)^{g_p^+(g_p^+-1)}\cdot H(x)^2\pmod{p}.\]
\end{thm}

\begin{note} From computational evidence, it appears that $H(x)$ is always coprime to $S_p(x)$, so that contrary to the situation on $X_0(p)$, only those supersingular points with quadratic irrational $j$-invariants correspond to Weierstrass points of $\X$. We give a heuristic argument for this phenomenon in Section \ref{fixedptssec}.\end{note}

In Section \ref{divisorssec} we start by reviewing some preliminary facts about divisors of polynomials of modular forms.  We then consider the reduction of $X_0(p)$ modulo $p$ in Section \ref{fixedptssec} in order to obtain a key result about the $w_p$-fixed points of $X_0(p)$. In Section \ref{basissec} we describe our good basis condition for $\mathcal{H}^1(\X)$. Next, in Section \ref{wronsksec} we derive a special cusp form on $\G_0(p)$ which encodes the Weierstrass weights of points on $\X$. In Section \ref{mainsec}, we prove Theorem \ref{main}, and in Section \ref{exsec}, we demonstrate Theorem \ref{main} for the curve $X_0^+(67)$.

\section{Divisor polynomials of modular forms}\label{divisorssec}

Let $M_k$ (resp. $M_k(p)$) denote the space of modular forms of weight $k$ on $\G$ (resp. $\G_0(p)$), and let $S_k$ (resp. $S_k(p)$) be the subspace of cusp forms.  For even $k\geq 4$, the Eisenstein series $E_k\in M_k$ is defined as
\[E_k(z):=1-\frac{2k}{B_k}\sum_{n=1}^\infty \sigma_{k-1}(n)q^n,\]
where $B_k$ is the $k$th Bernoulli number, and $\sigma_{k-1}(n)=\sum_{d\mid n}d^{k-1}$.  Then the function
\[\Delta(z):=\frac{E_4(z)^3-E_6(z)^2}{1728}=q-24q^2+252q^3-1472q^4+\cdots\]
is the unique normalized cusp form in $S_{12}$.

We briefly recall how to build a divisor polynomial whose zeros are exactly the $j$-values at which a given modular form $f\in M_k$ vanishes, excluding those zeros that may occur at the elliptic points $i$ and $\rho:=e^{2\pi i/3}$ (for details, see \cite{AO} or Section 2.6 of \cite{Ono}).  We define

\begin{eqnarray}\label{etilde}
\widetilde{E}_k(z):=\begin{cases}
1&\mbox{if $k\equiv 0\pmod{12}$},\\
E_4(z)^2E_6(z)&\mbox{if $k\equiv 2\pmod{12}$},\\
E_4(z)&\mbox{if $k\equiv 4\pmod{12}$},\\
E_6(z)&\mbox{if $k\equiv 6\pmod{12}$},\\
E_4(z)^2&\mbox{if $k\equiv 8\pmod{12}$},\\
E_4(z)E_6(z)&\mbox{if $k\equiv 10\pmod{12}$},
\end{cases}
\end{eqnarray}

and

\begin{eqnarray}\label{mk}
m(k):=\begin{cases}
\lfloor k/12\rfloor&\mbox{if $k\not\equiv 2\pmod{12}$},\\
\lfloor k/12\rfloor-1&\mbox{if $k\equiv 2\pmod{12}$}.
\end{cases}
\end{eqnarray}

Now let $f\in M_k$ have leading coefficient 1.  Then we note that (\ref{etilde}) and (\ref{mk}) are defined such that the quotient
\begin{eqnarray}\label{ftilde}\widetilde{F}(f,j(z)):=\frac{f(z)}{\Delta(z)^{m(k)}\widetilde{E}_k(z)}\end{eqnarray}
is a polynomial in $j(z)$.  Therefore, we define $\widetilde{F}(f,x)$ to be the unique polynomial in $x$ satisfying (\ref{ftilde}).
Furthermore, if $f$ has $p$-integral rational coefficients, then so does $\widetilde{F}(f,x)$.

Finally, we record a result about the divisor polynomial of the square of a modular form.

\begin{lem}\label{squaredivisor} Let $f\in M_k$. Then
\begin{eqnarray*}
\widetilde{F}(f^2,x)=\begin{cases}
\widetilde{F}(f,x)^2&\mbox{if $k\equiv 0\pmod{12}$},\\
x(x-1728)\widetilde{F}(f,x)^2&\mbox{if $k\equiv 2\pmod{12}$},\\
\widetilde{F}(f,x)^2&\mbox{if $k\equiv 4\pmod{12}$},\\
(x-1728)\widetilde{F}(f,x)^2&\mbox{if $k\equiv 6\pmod{12}$},\\
x\widetilde{F}(f,x)^2&\mbox{if $k\equiv 8\pmod{12}$},\\
(x-1728)\widetilde{F}(f,x)^2&\mbox{if $k\equiv 10\pmod{12}$}.
\end{cases}
\end{eqnarray*}
\end{lem}

\begin{proof} Using (\ref{ftilde}) for both $f$ and $f^2$ yields

\begin{equation*} f(z)^2=\Delta(z)^{2m(k)}\widetilde{E}_k(z)^2\widetilde{F}(f,j(z))^2,
\end{equation*}
and
\begin{equation*} f(z)^2=\Delta(z)^{m(2k)}\widetilde{E}_{2k}(z)\widetilde{F}(f^2,j(z)).
\end{equation*}

Thus
\begin{equation*}
\widetilde{F}(f^2,j(z))=\Delta(z)^{2m(k)-m(2k)}\cdot\frac{\widetilde{E}_k(z)^2}{\widetilde{E}_{2k}(z)}\cdot \widetilde{F}(f,j(z))^2.
\end{equation*}

Then by (\ref{etilde}) and (\ref{mk}) we have

\begin{eqnarray*}
\widetilde{F}(f^2,j(z))=\begin{cases}
\widetilde{F}(f,j(z))^2&\mbox{if $k\equiv 0\pmod{12}$},\\
\Delta(z)^{-2}E_4(z)^3E_6(z)^2\widetilde{F}(f,j(z))^2&\mbox{if $k\equiv 2\pmod{12}$},\\
\widetilde{F}(f,j(z))^2&\mbox{if $k\equiv 4\pmod{12}$},\\
\Delta(z)^{-1}E_6(z)^2\widetilde{F}(f,j(z))^2&\mbox{if $k\equiv 6\pmod{12}$},\\
\Delta(z)^{-1}E_4(z)^3\widetilde{F}(f,j(z))^2&\mbox{if $k\equiv 8\pmod{12}$},\\
\Delta(z)^{-1}E_6(z)^2\widetilde{F}(f,j(z))^2&\mbox{if $k\equiv 10\pmod{12}$},
\end{cases}
\end{eqnarray*}

Since $j(z)=\frac{E_4(z)^3}{\Delta(z)}$ and $j(z)-1728=\frac{E_6(z)^2}{\Delta(z)}$, the result follows.
\end{proof}

\section{Modular curves modulo $p$}\label{fixedptssec}

Here we recall the undesingularized reduction of $X_0(p)$ modulo $p$, due to Deligne and Rapoport \cite{DR}. The description below closely follows one given by Ogg \cite{OggRed}. The model of $X_0(p)$ modulo $p$ consists of two copies of $X_0(1)$ which meet transversally in the supersingular points (Figure \ref{Fig:x0p}).

\begin{figure}[b]\caption{Reduction of $X_0(p)$}\label{Fig:x0p}
\includegraphics{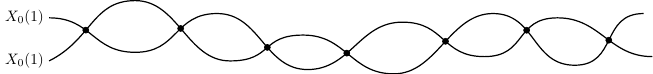}
\end{figure}

 The Atkin-Lehner operator $w_p$ is compatible with this reduction.  It gives an isomorphism between the two copies of $X_0(1)$ which preserves the supersingular locus, by fixing the points corresponding to supersingular curves defined over $F_p$, and interchanging those defined over $\F_{p^2}\backslash\F_p$ with their conjugates.  Therefore, dividing out by the action of $w_p$ glues together the two copies of $X_0(1)$. The singularities at the linear supersingular points are thus resolved, while the conjugate pairs of quadratic supersingular points are glued together. This results in a model for the reduction modulo $p$ of $X_0^+(p)$ consisting of one copy of $X_0(1)$ which self-intersects at each point representing a pair of conjugate quadratic supersingular points (Figure \ref{Fig:x0+p}). This resolution at the linear supersingular points may explain their absence among the Weierstrass points of $X_0^+(p)$.
 
\begin{figure}[h]\caption{Reduction of $X_0^+(p)$}\label{Fig:x0+p}
\includegraphics{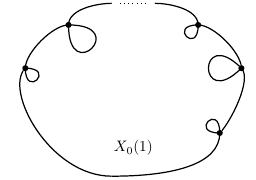}
\end{figure}

To make the correspondence between fixed points and linear supersingular $j$-invariants more precise, let $\mathcal{O}_D=\Z[\frac{1}{2}(D+\sqrt{-D})]$ be the order of the imaginary quadratic field $\Q[\sqrt{-D}]$ with discriminant $-D<0$.  The Hilbert class polynomial $\mathcal{H}_D(x)\in\Z[x]$ is the monic polynomial whose zeros are exactly the $j$-invariants of the distinct isomorphism classes of elliptic curves with complex multiplication by $\mathcal{O}_D$, and its degree is $h(-D)$, the class number of $\mathcal{O}_D$.

The points $Q\in Y_0(p)$ that are fixed by $w_p$ correspond to pairs $(E,C)$ such that $E$ admits complex multiplication by $\sqrt{-p}$, or in other words, $\Z[\sqrt{-p}]$ embeds in $\mbox{End}(E)$, the endomorphism ring of $E$ over the complex numbers (see e.g. \cite{OggHyp}).  Since $\mbox{End}(E)$ must be an order in an imaginary quadratic field, we have
\[\mbox{End}(E)\cong\begin{cases}\mathcal{O}_{4p}&\mbox{if }p\equiv 1\pmod{4},\\
\mathcal{O}_{p}\mbox{ or }\mathcal{O}_{4p}&\mbox{if }p\equiv 3\pmod{4}.\end{cases}\]
Now define 
\begin{eqnarray}\label{Hp}H_p(x):=\mathop{\prod_{\tau\in\G_0(p)\backslash\H}}_{v(Q_\tau)=2}(x-j(\tau)),\end{eqnarray}
the monic polynomial whose zeros are precisely the $j$-invariants of the $w_p$-fixed points of $Y_0(p)$.  Then we have
\[H_p(x)=\begin{cases}\mathcal{H}_{4p}(x)&\mbox{if }p\equiv 1\pmod{4},\\
\mathcal{H}_p(x)\cdot\mathcal{H}_{4p}(x)&\mbox{if }p\equiv 3\pmod{4}.\end{cases}\]
The following result is due independently to Kaneko and Zagier.
\begin{prop}\label{Hpsquare}\label{Kaneko} For $p$ prime, there exists a monic polynomial $T(x)\in \Z_p[x]$ with distinct roots such that $H_p(x)\equiv T(x)^2\pmod{p}$.\end{prop}

\begin{proof} The result follow from Kronecker's relations on the modular equation $\Phi_p(X,Y)$, and may be found in the appendix of \cite{Kan}.\end{proof}

We can now prove the following.

\begin{thm}\label{fixedlinear} Let $p$ be prime. Then we have
\[H_p(x)\equiv S_p^{(l)}(x)^2\pmod{p}.\]
\end{thm}

\begin{proof} 
By the discussion above, each zero of $H_p(x)$ is of the form $j(E)$ where $E$ is a supersingular elliptic curve defined over $\F_p$.  Then by Proposition \ref{Hpsquare}, $T(x)\mid S_p^{(l)}(x)$.  We will show that $T(x)$ and $S_p^{(l)}(x)$ have the same degree, proving that $T(x)=S_p^{(l)}(x)$. The result then follows again by Proposition \ref{Hpsquare}.

By the Riemann-Hurwitz formula (see, for example, Section I.2 of \cite{FK}), we have
\begin{eqnarray}\label{rh}2g_p^+=g_p+1-\frac{\sigma}{2},\end{eqnarray}
where $\sigma$ is the number of points of $X_0(p)$ at which the projection $\pi:X_0(p)\to X_0^+(p)$ is ramified, or in other words, the number of $w_p$-fixed points of $X_0(p)$.  We note that the cusps are not ramified since $w_p$ exchanges $0$ and $\infty$, so $\sigma=\mathrm{deg}(H_p(x))$.  On the other hand, Ogg explains in \cite{OggSC} that $g_p^+$ is equal to the number of conjugate pairs of supersingular $j$-invariants in $\F_{p^2}\backslash \F_p$.  Since there are $g_p+1$ total supersingular $j$-invariants, we have
\begin{eqnarray}\label{sssum}2g_p^+=g_p+1-\mathrm{deg}(S_p^{(l)}(x)).\end{eqnarray}
Then Proposition \ref{Kaneko}, (\ref{rh}) and (\ref{sssum}) imply that
\[\mathrm{deg}(T(x))=\frac{\mathrm{deg}(H_p(x))}{2}=\mathrm{deg}(S_p^{(l)}(x)).\]
\end{proof}

\section{A Good Basis for $\mathcal{H}^1(\X)$} \label{basissec}

For ease of notation we will let $g:=g_p^+$ for the rest of the paper, and assume that $g\geq 2$.
Recall that $g$ is the dimension of $\mathcal{H}^1(X_0^+(p))$, the space of holomorphic 1-forms on $\X$.  Let $\{\omega_1,\omega_2,\dots,\omega_g\}$ be a basis of $\mathcal{H}^1(X_0^+(p))$, where $\omega_i=h_i(u)du$ for some local variable $u$.  In order to take advantage of the correspondence that exists between holomorphic 1-forms on $X_0(p)$ and weight 2 cusp forms of level $p$, we pull back each $\omega_i$ to a holomorphic 1-form $\pi^*\omega_i$ on $X_0(p)$ via the projection map $\pi:X_0(p)\to\X$ (see, for example, Chapter 2 of \cite{Mir}). We can choose a local coordinate $z$ at $Q\in X_0(p)$ so that near $Q$, $u=z^n$, where $n$ is the multiplicity of $\pi$ at $\Q$, hence $n=v(Q)$ (\ref{vq}). Then we have $\pi^*\omega_i=H_i(z)dz$ with $H_i(z)=h_i(z^n)nz^{n-1}\in S_2(p)$. Since each $H_i(z)$ has been pulled back from $\X$, it must be invariant under $w_p$, so it is a member of $S_2^+(p)$, the subspace of $w_p$-invariant cusp forms of weight $2$.  In fact, it is straightforward to show that $\{H_1(z),H_2(z),\dots,H_g(z)\}$ forms a basis for $S_2^+(p)$.  


It will be helpful later on to specify a basis for $S_2^+(p)$ of a particularly nice form. First, we can guarantee a basis with rational Fourier coefficients by the following argument. The space $S_2(p)$ has a basis consisting of newforms. Let $f(z)=\sum_n a(n)q^n$ be a newform for $S_2(p)$, and let $\sigma\in \mbox{Gal}(\C/\Q)$. Then $f^\sigma(z)=\sum_n \sigma(a(n))q^n$ is also a newform for $S_2(p)$, so the action of $\mbox{Gal}(\C/\Q)$ partitions the newforms into Galois conjugacy classes. If two newforms are Galois conjugates then they share the same eigenvalue for $w_p$. Let $V_f$ be the $\C$-vector space spanned by the Galois conjugates of $f$. Standard Galois-theoretic arguments show that $V_f$ has a basis consisting of cusp forms with rational coefficients. These are no longer newforms, but as they are linear combinations of the Galois conjugates of $f$, they are still eigenforms for $w_p$. Therefore collecting such a basis for each Galois conjugacy class with eigenvalue $1$ for $w_p$ yields a basis for $S_2^+(p)$ with rational Fourier coefficients.

We can determine such a basis $\{f_1,f_2,\dots,f_g\}$ uniquely by requiring that
\begin{align}\label{basis}
f_1(z)&=q^{c_1}+O(q^{c_g+1})\\
f_2(z)&=q^{c_2}+O(q^{c_g+1})\notag\\
&\vdots\notag\\
f_g(z)&=q^{c_g}+O(q^{c_g+1})\notag
\end{align}
where 
\begin{equation}\label{handy}
c_1<c_2<\dots<c_g.
\end{equation} 

\begin{definition} We say that $\mathcal{H}^1(X_0^+(p))$ has a \textit{good basis} if the cusp forms $f_1,f_2,\dots,f_g$ satisfying (\ref{basis}) and (\ref{handy}) have $p$-integral Fourier coefficients.\end{definition}

\section {Wronskians and $p$-integrality} \label{wronsksec}

Given any basis $\{\omega_1,\omega_2,\dots,\omega_g\}$ for $\mathcal{H}^1(X_0^+(p))$ with $\omega_i=h_i(u)du$, we define the Wronskian
\begin{eqnarray}\label{wronsk}W(h_1,h_2,\dots,h_g)(u):=
\left|\begin{array}{cccc}h_1&h_2&\cdots&h_g\\
 h_1'&h_2'&\cdots&h_g'\\
 \vdots&\vdots&\vdots&\vdots\\
 h_1^{(g-1)}&h_2^{(g-1)}&\cdots&h_g^{(g-1)}
 \end{array}\right|.\end{eqnarray}

Let $\mathcal{W}^+(u)$ be the scalar multiple of $W(h_1,h_2,\dots,h_g)(u)$ with leading coefficient 1, so that $\mathcal{W}^+(u)$ is independent of the choice of basis. It is well-known that the Wronskian encodes the Weierstrass weights of points in $\X$ (see \cite{FK}, page 82).  Specifically,
\[\wt(\overline{Q})=\ord_{\overline{Q}}(\mathcal{W}^+(u)(du)^{g(g+1)/2}).\]
Since it is advantageous to work on $X_0(p)$ instead of $\X$, we consider the pullback of $W^+:=\mathcal{W}^+(u)(du)^{g(g+1)/2}$ to $X_0(p)$ via $\pi$, which is $\pi^*W^+=\mathcal{W}^+(z^n)(nz^{n-1}dz)^{g(g+1)/2}$. Recalling that $n=v(Q)$ when $z$ is near $Q$, we have

\begin{eqnarray}\label{ordWplus}
\ord_Q(\pi^*W^+)=v(Q)\wt(\overline{Q})+\frac{g(g+1)}{2}(v(Q)-1).
\end{eqnarray}

Alternatively, we could pull back each $\omega_i$ individually to $\pi^*\omega_i=H_i(z)dz$ as in Section \ref{basissec}. Then we can form the Wronskian $W(H_1,H_2,\dots,H_g)(z)$ (defined analogously to (\ref{wronsk})). Since the $H_i$ are cusp forms of weight $2$ for $\G_0(p)$ then $W(H_1,H_2,\dots,H_g)(z)$ is a cusp form of weight $g(g+1)$ for $\G_0(p)$. It can be shown using basic facts about determinants that
\[W(H_1,H_2,\dots,H_g)(z)(dz)^{g(g+1)/2}
=W(h_1,h_2,\dots,h_g)(z^n)(nz^{n-1}dz)^{g(g+1)/2}.\]
Now let $\mathcal{W}_p(z)$ be the multiple of $W(H_1,H_2,\dots,H_g)(z)$ with leading coefficient 1. Then $\mathcal{W}_p(z)$ is independent of the choice of basis for $S_2^+(p)$, and we have $\mathcal{W}_p(z)(dz)^{g(g+1)/2}=\pi^* W^+$, hence by (\ref{ordWplus}),

\begin{eqnarray}\label{ordWp}
\ord_Q(\mathcal{W}_p(z)(dz)^{g(g+1)/2})=v(Q)\wt(\overline{Q})+\frac{g(g+1)}{2}(v(Q)-1).
\end{eqnarray}

We next see the advantage of having a good basis for $\mathcal{H}^1(X_0^+(p))$.

\begin{thm}\label{goodbasis} Let $p$ be a prime such that $\mathcal{H}^1(X_0^+(p))$ has a good basis. Then $\mathcal{W}_p(z)\in S_{g(g+1)}(p)$ has $p$-integral rational coefficients.\end{thm}

\begin{proof} Here we closely follow the proof of Lemma 3.1 in \cite{AP}. Let $\{f_1,f_2,\dots,f_g\}$ be a basis for $S_2^+(p)$ satisfying (\ref{basis}) and (\ref{handy}). Let $\theta:=q\frac{d}{dq}$ be the usual differential operator for modular forms, so that $\frac{d}{dz}=2\pi i\theta$. Then by properties of determinants, we have

\begin{equation*}W(f_1,f_2,\dots,f_g)=(2\pi i)^{g(g-1)/2}\left|\begin{array}{cccc}f_1&f_2&\cdots&f_g\\
 \theta f_1&\theta f_2&\cdots&\theta f_g\\
 \vdots&\vdots&\vdots&\vdots\\
 \theta f_1^{(g-1)}&\theta f_2^{(g-1)}&\cdots&\theta f_g^{(g-1)}
 \end{array}\right|.\end{equation*}
We see that the Fourier expansion of $\left(\frac{1}{2\pi i}\right)^{g(g-1)/2}W(f_1,f_2,\dots,f_g)$ has rational $p$-integral coefficients, with leading coefficient given by the Vandermonde determinant

\begin{equation}\label{Vand} V:=\left|\begin{array}{cccc}1&1&\cdots&1\\
 c_1&c_2&\cdots&c_g\\
 \vdots&\vdots&\vdots&\vdots\\
 c_1^{(g-1)}&c_2^{(g-1)}&\cdots&c_g^{(g-1)}
 \end{array}\right|=\prod_{1\leq j<k\leq g} (c_k-c_j).\end{equation}
It now suffices to show that $p$ does not divide the leading coefficient. By Sturm's bound \cite{Sturm} for the order of vanishing modulo $p$ for modular forms on $\G_0(p)$, we have $1\leq c_i\leq \frac{p+1}{6}<p$ for each $1\leq i\leq g$, so $1\leq c_k-c_j \leq p-1$ for all $j<k$. Therefore the lemma is proved.
\end{proof}

\section{Proof of the Main Theorem}\label{mainsec}

Let $p$ be a prime for which $\mathcal{H}^1(X_0^+(p))$ has a good basis. We note that when $g<2$, there are no Weierstrass points on $X_0^+(p)$.  Then $\mathcal{F}_p(x)=1$ and $g^2-g=0$, so the theorem holds trivially by taking $H(x)=1$.  Thus from here on we will assume that $g\geq 2$, in which case we have $p\geq 67$.  

We first adapt two lemmas from \cite{AO}. For any meromorphic function $f(z)$ defined on $\H$ and any integer $k$, we define the \textit{slash} operator $\mid_k$ by 
\[f(z)|_k\gamma:=(\det\gamma)^{k/2}(cz+d)^{-k}f(\gamma z),\]
where $\gamma:=\left(\begin{smallmatrix}a&b\\c&d\end{smallmatrix}\right)$ is a real matrix with positive determinant, and $\gamma z:=\frac{az+b}{cz+d}$. In particular, the Atkin-Lehner involution $w_p$ is given by $f\mapsto f|_k \left(\begin{smallmatrix}0&-1\\p&0\end{smallmatrix}\right)$ when $f$ is a modular form of weight $k$.

\begin{lem}\label{AOlem3.2} We have 
\[\mathcal{W}_p(z)|_{g(g+1)}\left(\begin{smallmatrix}0&-1\\p&0\end{smallmatrix}\right)=\mathcal{W}_p(z). \]
\end{lem}

\begin{proof} The proof is identical to Lemma 3.2 of \cite{AO} except that $f|_2\left(\begin{smallmatrix}0&-1\\p&0\end{smallmatrix}\right)=f$ for every newform $f$ in $S_2^+(p)$.\end{proof}

\begin{lem}\label{wpsquared} If $p$ is a prime such that $X_0^+(p)$ has genus at least 2, define
\[\widetilde{\mathcal{W}}_p(z):=\prod_{A\in\G_0(p)\backslash\G}\mathcal{W}_p(z)|_{g(g+1)}A,\]
normalized to have leading coefficient 1. Then $\widetilde{\mathcal{W}}_p(z)$ is a cusp form of weight $g(g+1)(p+1)$ on $\G$ with $p$-integral rational coefficients, and
\[\widetilde{\mathcal{W}}_p(z)\equiv \mathcal{W}_p(z)^2\pmod{p}.\]
\end{lem}

\begin{proof} This follows from our Lemma \ref{AOlem3.2} exactly as Lemma 3.3 follows from Lemma 3.2 in \cite{AO}.\end{proof}

We again consider a basis $\{f_1,f_2,\dots,f_g\}$ for $S_2^+(p)$ satisfying (\ref{basis}) and (\ref{handy}). For each $f_i$, there is a cusp form $b_i\in S_{p+1}$ with $p$-integral rational  coefficients for which $f_i\equiv b_i\pmod{p}$ (\cite{AP}, Theorem 4.1(c)).  Define $W(z)$ to be the multiple of $W(b_1,b_2,\dots,b_g)$ with leading coefficient 1.  By the same reasoning as in Theorem \ref{goodbasis}, $\left(\frac{1}{2\pi i}\right)^{g(g-1)/2}W(b_1,b_2,\dots,b_g)$ has $p$-integral rational coefficients and leading coefficient $V$ (\ref{Vand}). Since the differential operator $\theta$ preserves congruences, we have
\begin{equation*}
\left(\frac{1}{2\pi i}\right)^{g(g-1)/2}W(f_1,f_2,\dots,f_g)\equiv \left(\frac{1}{2\pi i}\right)^{g(g-1)/2}W(b_1,b_2,\dots,b_g)\pmod{p},\end{equation*}
and hence
\begin{equation*}
V\cdot\mathcal{W}_p(z)\equiv V\cdot W(z)\pmod{p}.
\end{equation*}
Since $V$ is coprime to $p$, then by Lemma \ref{wpsquared} we have 
\[\widetilde{\mathcal{W}}_p(z)\equiv \mathcal{W}_p(z)^2\equiv W(z)^2\pmod{p}.\]
We now have two cusp forms $\widetilde{\mathcal{W}}_p(z)$ and $W(z)^2$ on the full modular group, but $\widetilde{\mathcal{W}}_p(z)$ has weight $\tilde{k}(p):=g(g+1)(p+1)$ while $W(z)^2$ has weight $2g(g+p)$.  Using the fact that the Eisenstein series $E_{p-1}(z)\equiv 1\pmod{p}$, we have
\begin{eqnarray}\label{wpws}\widetilde{\mathcal{W}}_p(z)\equiv W(z)^2\cdot E_{p-1}(z)^{g^2-g}\pmod{p},\end{eqnarray}
where the cusp forms on each side of the congruence in (\ref{wpws}) have the same weight $\tilde{k}(p)$.  By (\ref{ftilde}) there exist polynomials $\widetilde{F}(\widetilde{\mathcal{W}}_p(x),x)$ and $\widetilde{F}(W^2E_{p-1}^{g^2-g},x)$ with $p$-integral rational coefficients such that
\[\widetilde{\mathcal{W}}_p(z)=\Delta(z)^{m(\tilde{k}(p))}\widetilde{E}_{\tilde{k}(p)}(z)\widetilde{F}(\widetilde{\mathcal{W}}_p,j(z))\]
and
\[W(z)^2E_{p-1}(z)^{g^2-g}=\Delta(z)^{m(\tilde{k}(p))}\widetilde{E}_{\tilde{k}(p)}(z)\widetilde{F}(W^2E_{p-1}^{g^2-g},j(z)).\]
Then by (\ref{wpws}), we conclude that
\begin{eqnarray}\label{ftildecong}\widetilde{F}(\widetilde{\mathcal{W}}_p,x)\equiv \widetilde{F}(W^2 E_{p-1}^{g^2-g},x)\pmod{p}.\end{eqnarray}

We next compute each side of (\ref{ftildecong}).  To compute the right hand side, we begin with the following.

\begin{lem}\label{cp} \emph{(Theorem 2.3 in \cite{AO})} For a prime $p\geq 5$ and $f\in M_k$ with $p$-integral coefficients, we have
\[\widetilde{F}(fE_{p-1},x)\equiv \widetilde{F}(E_{p-1},x)\cdot\widetilde{F}(f,x)\cdot C_p(k;x)\pmod{p}\]
where
\[C_p(k;x):=\begin{cases}x&\mbox{if $(k,p)\equiv(2,5)$, $(8,5)$, $(8,11)\pmod{12}$},\\
x-1728&\mbox{if $(k,p)\equiv(2,7)$, $(6,7)$, $(10,7)$, $(6,11)$, $(10,11) \pmod{12}$},\\
x(x-1728)&\mbox{if $(k,p)\equiv(2,11)\pmod{12}$},\\
1&\mbox{otherwise}.\end{cases}\]
\end{lem}

Then using Lemma \ref{cp} inductively, we have
\[\widetilde{F}(W^2\cdot E_{p-1}^{g^2-g},x)\equiv \widetilde{F}(E_{p-1},x)^{g^2-g}\cdot\widetilde{F}(W^2,x)\cdot \mathcal{G}_p(x)\pmod{p},\]
where
\[\mathcal{G}_p(x):=\prod_{s=1}^{g^2-g}C_p(2g(g+p)+(g^2-g-s)(p-1);x).\]
A case-by-case computation reveals that
\[\mathcal{G}_p(x)=\begin{cases}1&\mbox{if $p\equiv 1\pmod{12}$},\\
x^{\lceil\frac{g^2-g}{3}\rceil}&\mbox{if $p\equiv 5\pmod{12}$},\\
(x-1728)^{(g^2-g)/2}&\mbox{if $p\equiv 7\pmod{12}$},\\
x^{\lceil\frac{g^2-g}{3}\rceil}(x-1728)^{(g^2-g)/2}&\mbox{if $p\equiv 11\pmod{12}$}.\end{cases}\]
By a result of Deligne (see \cite{SerreHPF}), and recalling (\ref{Stilde}), we have 
\[\widetilde{F}(E_{p-1},x)\equiv \widetilde{S}_p(x)\pmod{p},\]
and therefore
\begin{eqnarray}\label{rhs}\widetilde{F}(W^2 E_{p-1}^{g^2-g},x)\equiv \widetilde{S}_p(x)^{g^2-g}\cdot\widetilde{F}(W^2,x)\cdot \mathcal{G}_p(x)\pmod{p}.\end{eqnarray}

Next we evaluate the left-hand side of (\ref{ftildecong}). We recall here the definitions

\[\mathcal{F}_p(x):=\prod_{Q\in Y_0(p)}(x-j(Q))^{v(Q)\wt(\overline{Q})},\]
and
\[H_p(x):=\mathop{\prod_{\tau\in\G_0(p)\backslash\H}}_{v(Q_\tau)=2}(x-j(\tau)).\]

\begin{thm}\label{ftildesplit} Let $p$ be a prime such that the genus of $X_0^+(p)$ is at least 2. Define $\epsilon_p(i)$ and $\epsilon_p(\rho)$ by
\[\epsilon_p(i)=\frac{(g^2+g)\left(1+\leg{-1}{p}\right)}{4},\]
and
\[\epsilon_p(\rho)=\frac{(g^2+g)\left(1+\leg{-3}{p}\right)-k^*}{3},\]
where $k^*\in\{0,1,2\}$ with $k^*\equiv \tilde{k}(p)\pmod{3}$.  Then we have
\[\widetilde{F}(\widetilde{\mathcal{W}}_p,x)=x^{\epsilon_p(\rho)}(x-1728)^{\epsilon_p(i)}\mathcal{F}_p(x)H_p(x)^{g(g+1)/2}.\]
\end{thm}

\begin{proof} If $\tau_0\in\H$ and $A\in\G$, then
\[\ord_{\tau_0}(\mathcal{W}_p(z)\mid_{g(g+1)}A)=\ord_{A(\tau_0)}(\mathcal{W}_p(z)),\]
so that
\begin{equation}\label{ordtau}\ord_{\tau_0}(\widetilde{\mathcal{W}}_p(z))=\sum_{A\in\G_0(p)\backslash\G}\ord_{A(\tau_0)}(\mathcal{W}_p(z)).\end{equation}
Now recall by (\ref{ordWp}) that for $Q\in Y_0(p)$, we have
\[\ord_Q(\mathcal{W}_p(z)(dz)^{g(g+1)/2})
=v(Q)\wt(\overline{Q})+\frac{g(g+1)}{2}(v(Q)-1).\]
Let $\ell_\tau\in\{1,2,3\}$ be the order of the isotropy subgroup of $\tau$ in $\G_0(p)/\{\pm I\}$, where $\tau$ is an elliptic fixed point if and only if $\ell(\tau)\ne 1$.  If $Q_\tau\in Y_0(p)$ is associated to $\tau\in\H$ in the usual way then we have
\begin{align}\label{ordwp}
\ord_\tau(\mathcal{W}_p(z))&=\ell_\tau\ord_{Q_\tau}(\mathcal{W}_p(z)(dz)^{g(g+1)/2})+\frac{g(g+1)}{2}(\ell_\tau-1)\\
&=\ell_\tau v(Q_\tau)\wt(\overline{Q_\tau})+\frac{g(g+1)}{2}(\ell_\tau v(Q_\tau)-1).\notag
\end{align}
If $\tau_0$ is not equivalent to $i$ or $\rho$ under $\G$, then $\{A(\tau_0)\}_{A\in\G_0(p)\backslash\G}$ consists of $p+1$ points which are $\G_0(p)$-inequivalent, so by (\ref{ordtau}) and (\ref{ordwp}),
\[\ord_{\tau_0}(\widetilde{\mathcal{W}}_p(z))=\mathop{\sum_{\tau\in\G_0(p)\backslash\H}}_{\tau\stackrel{\G}{\sim}\tau_0}
\ord_\tau(\mathcal{W}_p(z))=\mathop{\sum_{\tau\in\G_0(p)\backslash\H}}_{\tau\stackrel{\G}{\sim}\tau_0}
v(Q_\tau)\wt(\overline{Q_\tau})+\frac{g(g+1)}{2}(v(Q_\tau)-1).\]

When $\tau_0\stackrel{\G}{\sim}\rho$, then $\ord_{\tau_0}(\widetilde{\mathcal{W}}_p(z))=\ord_{\rho}(\widetilde{\mathcal{W}}_p(z))$, and  $\{A(\rho)\}_{A\in\G_0(p)\backslash\G}$ contains $1+\leg{-3}{p}$ elliptic fixed points of order 3 which are $\G_0(p)$-inequivalent, and $p-\leg{-3}{p}$ additional points which are partitioned into $\G_0(p)$-orbits of size 3.  Then by (\ref{ordwp}) we have
\begin{multline}\label{ordrho}
\ord_{\rho}(\widetilde{\mathcal{W}}_p(z))=
3\mathop{\sum_{\tau\in\G_0(p)\backslash\H}}_{\tau\stackrel{\G}{\sim}\rho,\;\ell(\tau)=1}
\ord_\tau(\mathcal{W}_p(z))
+\mathop{\sum_{\tau\in\G_0(p)\backslash\H}}_{\tau\stackrel{\G}{\sim}\rho,\;\ell(\tau)=3}
\ord_\tau(\mathcal{W}_p(z))\\
=3\mathop{\sum_{\tau\in\G_0(p)\backslash\H}}_{\tau\stackrel{\G}{\sim}\rho,\;\ell(\tau)=1}
\left(v(Q_\tau)\wt(\overline{Q_\tau})+\frac{g(g+1)}{2}(v(Q_\tau)-1)\right)\\
+\mathop{\sum_{\tau\in\G_0(p)\backslash\H}}_{\tau\stackrel{\G}{\sim}\rho,\;\ell(\tau)=3}
\left(3v(Q_\tau)\wt(\overline{Q_\tau})+\frac{g(g+1)}{2}(3v(Q_\tau)-1)\right)\\
=3\left(\mathop{\sum_{\tau\in\G_0(p)\backslash\H}}_{\tau\stackrel{\G}{\sim}\rho}
v(Q_\tau)\wt(\overline{Q_\tau})+\frac{g(g+1)}{2}(v(Q_\tau)-1)\right)
+(g^2+g)\left(1+\leg{-3}{p}\right).
\end{multline}

When $\tau_0\stackrel{\G}{\sim}i$, then $\ord_{\tau_0}(\widetilde{\mathcal{W}}_p(z))=\ord_{i}(\widetilde{\mathcal{W}}_p(z))$, and  $\{A(i)\}_{A\in\G_0(p)\backslash\G}$ contains $1+\leg{-1}{p}$ elliptic fixed points of order 2 which are $\G_0(p)$-inequivalent, and $p-\leg{-1}{p}$ additional points which are partitioned into $\G_0(p)$-orbits of size 2.  We then have
\begin{multline}\label{ordi}
\ord_{i}(\widetilde{\mathcal{W}}_p(z))=
2\mathop{\sum_{\tau\in\G_0(p)\backslash\H}}_{\tau\stackrel{\G}{\sim}i,\;\ell(\tau)=1}
\ord_\tau(\mathcal{W}_p(z))
+\mathop{\sum_{\tau\in\G_0(p)\backslash\H}}_{\tau\stackrel{\G}{\sim}i,\;\ell(\tau)=2}
\ord_\tau(\mathcal{W}_p(z))\\
=2\mathop{\sum_{\tau\in\G_0(p)\backslash\H}}_{\tau\stackrel{\G}{\sim}i,\;\ell(\tau)=1}
\left(v(Q_\tau)\wt(\overline{Q_\tau})+\frac{g(g+1)}{2}(v(Q_\tau)-1)\right)\\
+\mathop{\sum_{\tau\in\G_0(p)\backslash\H}}_{\tau\stackrel{\G}{\sim}i,\;\ell(\tau)=2}
\left(2v(Q_\tau)\wt(\overline{Q_\tau})+\frac{g(g+1)}{2}(2v(Q_\tau)-1)\right)\\
=2\left(\mathop{\sum_{\tau\in\G_0(p)\backslash\H}}_{\tau\stackrel{\G}{\sim}i}
v(Q_\tau)\wt(\overline{Q_\tau})+\frac{g(g+1)}{2}(v(Q_\tau)-1)\right)
+\frac{g^2+g}{2}\left(1+\leg{-1}{p}\right).
\end{multline}
Finally, we recall that $j(z)$ vanishes to order 3 at $z=\rho$, that $j(z)-1728$ vanishes to order 2 at $z=i$, and that $j(z)-j(\tau_0)$ vanishes to order 1 at all other points $\tau_0\in\G\backslash\H$.  Therefore the exponent of $x-j(\tau_0)$ in $\widetilde{F}(\widetilde{\mathcal{W}}_p,x)$ is equal to 
\begin{eqnarray}\label{exponents}\begin{cases}\ord_{\tau_0}\widetilde{\mathcal{W}}_p&\mbox{if $\tau_0\ne i,\rho$},\\
\frac{1}{2}\ord_i\widetilde{\mathcal{W}}_p&\mbox{if $\tau_0=i$},\\
\frac{1}{3}(\ord_i\widetilde{\mathcal{W}}_p-k^*)&\mbox{if $\tau_0=\rho$}.\end{cases}\end{eqnarray}

Therefore by (\ref{Hp}), (\ref{ordrho}), (\ref{ordi}) and (\ref{exponents}) we have
\begin{align*}\widetilde{F}(\widetilde{\mathcal{W}}_p,x)
&=x^{\epsilon_p(\rho)}(x-1728)^{\epsilon_p(i)}\mathcal{F}_p(x)\mathop{\prod_{\tau\in\G_0(p)\backslash\H}}_{v(Q_\tau)=2}(x-j(\tau))^{g(g+1)/2}\\
&=x^{\epsilon_p(\rho)}(x-1728)^{\epsilon_p(i)}\mathcal{F}_p(x)H_p(x)^{g(g+1)/2}.\end{align*}
\end{proof}

Combining (\ref{ftildecong}), (\ref{rhs}), Theorem \ref{fixedlinear} and Theorem \ref{ftildesplit} now yields
\begin{equation}\label{cong} 
x^{\epsilon_p(\rho)}(x-1728)^{\epsilon_p(i)}\mathcal{F}_p(x)S_p^{(l)}(x)^{g^2+g}\equiv  \widetilde{S}_p(x)^{g^2-g}\cdot\widetilde{F}(W^2,x)\cdot \mathcal{G}_p(x)\pmod{p}.\end{equation}

We next define
\[\widetilde{S}_p^{(l)}(x):=\mathop{\prod_{E/\overline{\F}_p\;\mathrm{supersingular}}}_{j(E)\in\F_p\backslash\{0,1728\}} (x-j(E)).\]
In the chart below, we compare certain factors appearing in (\ref{cong}) for each choice of $p$ modulo $12$.  

\begin{table}[h]
\caption{Factors arising from elliptic points}
\label{table}
\centering
\[\begin{array}{|c|cccc|}
\hline
p\mod{12}&x^{\epsilon_p(\rho)}&(x-1728)^{\epsilon_p(i)}&\mathcal{G}_p(x)&S_p^{(l)}(x)\\
\hline
1&x^{\lfloor\frac{2(g^2+g)}{3}\rfloor}&(x-1728)^{(g^2+g)/2}&1&\widetilde{S}_p^{(l)}(x)\\
5&1&(x-1728)^{(g^2+g)/2}&x^{\lceil\frac{g^2-g}{3}\rceil}&x\cdot\widetilde{S}_p^{(l)}(x)\\
7&x^{\lfloor\frac{2(g^2+g)}{3}\rfloor}&1&(x-1728)^{(g^2-g)/2}&(x-1728)\cdot\widetilde{S}_p^{(l)}(x)\\
11&1&1&x^{\lceil\frac{g^2-g}{3}\rceil}(x-1728)^{(g^2-g)/2}&x(x-1728)\cdot\widetilde{S}_p^{(l)}(x)\\
\hline
\end{array}\]

\end{table}
\bsk

Since both $\lceil\frac{g^2-g}{3}\rceil$ and $\frac{g^2-g}{2}$ are less than $g^2+g$, we see from Table \ref{table} that $\mathcal{G}_p(x)$ always divides $S_p^{(l)}(x)^{g^2+g}$. Then since $x$ and $(x-1728)$ are coprime to $\widetilde{S}_p(x)$ we have
\begin{equation}\label{firstcancel}\mathcal{F}_p(x)\frac{S_p^{(l)}(x)^{g^2+g}}{\mathcal{G}_p(x)}\equiv \widetilde{S}_p(x)^{g^2-g}\frac{\widetilde{F}(W^2,x)}{x^{\epsilon_p(\rho)}(x-1728)^{\epsilon_p(i)}}\pmod{p},\end{equation}
where the two quotients reduce to polynomials.

Now on the left, we write $S_p^{(l)}(x)=x^{\alpha_p(\rho)}(x-1728)^{\alpha_p(i)}\widetilde{S}_p^{(l)}(x)$ with $\alpha_p(\rho),\alpha_p(i)\in\{0,1\}$ according to $p$ modulo 12, as in Table \ref{table}. On the right, we write $\widetilde{S}_p(x)=\widetilde{S}_p^{(l)}(x)S^{(q)}(x)$. Then (\ref{firstcancel}) becomes

\begin{equation}\label{secondcancel}\mathcal{F}_p(x)\widetilde{S}_p^{(l)}(x)^{g^2+g}\frac{(x^{\alpha_p(\rho)}(x-1728)^{\alpha_p(i)})^{g^2+g}}{\mathcal{G}_p(x)}\equiv \widetilde{S}^{(l)}_p(x)^{g^2-g}S_p^{(q)}(x)^{g^2-g}\frac{\widetilde{F}(W^2,x)}{x^{\epsilon_p(\rho)}(x-1728)^{\epsilon_p(i)}}\pmod{p}.\end{equation}

Now the quotient on the left of (\ref{secondcancel}) must divide $\widetilde{F}(W^2,x)$. Then cancelling $\widetilde{S}^{(l)}_p(x)^{g^2-g}$ on each side leaves $\widetilde{S}_p^{(l)}(x)^{2g}$ on the left, which  must then divide $\widetilde{F}(W^2,x)$ as well. So (\ref{secondcancel}) becomes

\[\mathcal{F}_p(x)\equiv S_p^{(q)}(x)^{g^2-g}H_1(x)\pmod{p},\]

where $H_1(x)$ is the polynomial given in non-reduced form by the quotient

\[H_1(x):=\frac{\mathcal{G}_p(x)\widetilde{F}(W^2,x)}{x^{\epsilon_p(\rho)}(x-1728)^{\epsilon_p(i)}(x^{\alpha_p(\rho)}(x-1728)^{\alpha_p(i)})^{g^2+g}\widetilde{S}_p^{(l)}(x)^{2g}}.\]

It remains to show that $H_1(x)$ is a perfect square. By Lemma \ref{squaredivisor}, we write $\widetilde{F}(W^2,x)=x^{\delta_p(\rho)}(x-1728)^{\delta_p(i)}\widetilde{F}(W,x)^2$, where $\delta_p(\rho),\delta_p(i)\in\{0,1\}$ according to $g(g+p)$ modulo 12. We then decompose $H_1(x)$ into a product of two quotients,

\[H_1(x)=\frac{\mathcal{G}_p(x)x^{\delta_p(\rho)}(x-1728)^{\delta_p(i)}}{x^{\epsilon_p(\rho)}(x-1728)^{\epsilon_p(i)}}\cdot \frac{\widetilde{F}(W,x)^2}{(x^{\alpha_p(\rho)}(x-1728)^{\alpha_p(i)})^{g^2+g}\widetilde{S}_p^{(l)}(x)^{2g}}.\]
Note that the exponents in the right-hand quotient are all even. The quotient on the left is of the form $x^a(x-1728)^b$, where $a$ and $b$ are integers, possibly negative. It is sufficient to show that $a$ and $b$ are both even. An examination of the exponents reveals that the parity of $a$ and $b$ depend only on $p$ and $g$ modulo 12. A check of all possible combinations of these values using Table \ref{table} and Lemma \ref{squaredivisor} confirms that $a$ and $b$ are indeed even in all cases, and therefore we can write $H_1(x)=H(x)^2$ for some polynomial $H(x)\in\mathbb{F}_p$. This concludes the proof of Theorem \ref{main}. \qed

\section{The example for $X_0^+(67)$}\label{exsec}

Here we compute $\mathcal{F}_{67}(x)$, the divisor polynomial corresponding to the modular curve $X_0^+(67)$, which has genus 2.   A basis for $S_2^+(67)$ is given by $\{f_1,f_2\}$, with
\[f_1=q - 3q^3 - 3q^4 - 3q^5 + q^6 + 4q^7 + 3q^8 + \cdots,\]
and
\[f_2=q^2 - q^3 - 3q^4 + 3q^7 + 4q^8 + \cdots.\]
The associated Wronskian is  
\[\mathcal{W}_{67}(z)=q^3 - 2q^4 - 6q^5 + 6q^6 + 15q^7 + 8q^8 +\cdots\in S_6(67).\]
Then by Lemma \ref{wpsquared} and (\ref{ftilde}), we have
\[\widetilde{F}(\widetilde{\mathcal{W}}_{67},x)\equiv x^4(x+1)^6(x+14)^6(x^2+8x+45)^2(x^2+44x+24)^2(x^2+10x+62)^2\pmod{67}.\]
But $\epsilon_{67}(i)=0$, $\epsilon_{67}(\rho)=4$, and
\[S_{67}(x)=(x+1)(x+14)(x^2+8x+45)(x^2+44x+24).\]
Therefore by Theorem \ref{ftildesplit} we have
\begin{align*}
\mathcal{F}_{67}(x)
&\equiv (x^2+8x+45)^2(x^2+44x+24)^2(x^2+10x+62)^2\pmod{67}\\
&\equiv S_{67}^{(q)}(x)^2(x^2+10x+62)^2\pmod{67}.
\end{align*}

\begin{note} In general, $H(x)$ may not be irreducible.\end{note}

\section{Acknowledgements}

The author gratefully acknowledges Scott Ahlgren for his invaluable mentoring and for suggesting this problem in the first place.

\end{document}